\documentclass[11pt]{amsart}
\usepackage{geometry}                
\usepackage{graphicx}
\usepackage{amssymb}
\usepackage{epstopdf}
\usepackage{amsthm}
\usepackage{thmtools}
\usepackage{amsmath}
\usepackage{amscd}
\usepackage{verbatim}
\usepackage{amscd}
\usepackage{hyperref}
\usepackage{lscape}
\usepackage{comment}
\usepackage{cite}
\DeclareMathOperator{\Spec}{Spec}
\DeclareMathOperator{\coker}{coker}
\DeclareMathOperator{\im}{im}

\DeclareMathOperator{\p}{\mathfrak{p}}

\DeclareMathOperator{\kk}{\mathbb K}
\newcommand{\tildemid}{\raisebox{0.5ex}{\texttildelow}}

\newtheorem{theorem}{Theorem}[section]
\newtheorem{definition}[theorem]{Definition}

\newtheorem{lemma}[theorem]{Lemma}
\newtheorem{prop}[theorem]{Proposition}

\begin{document}
\title{On the de Rham homology of affine varieties in characteristic 0}
\author{Nicole Bridgland}
\thanks{NSF support through grants DMS-1500264 and DMS-1800355 is gratefully acknowledged}
\address{Department of Mathematics, University of Minnesota, Minneapolis, MN 55455}
\email{bridg130@umn.edu}
\maketitle

\begin{abstract}
Let $\kk$ be a field of characteristic 0, let $S$ be a complete local ring with coefficient field $\kk$, let $\kk[[x_1,\dots,x_n]]$ be the ring of formal power series in variables $x_1,\dots, x_n$ with coefficients from $\kk$, let $\kk[[x_1,\dots,x_n]]\to S$ be a $\kk$-algebra surjection and let $E_\bullet^{\bullet,\bullet} $ be the associated Hodge-de Rham spectral sequence for the computation of the de Rham homology of $S$. Nicholas Switala \cite{switala} proved that this spectral sequence is independent of the surjection beginning with the $E_2$ page, and the groups $E^{p,q}_2$ are all finite-dimensional over $\kk$. 

In this paper we extend this result to affine varieties. Namely, let $Y$ be an affine variety over $\kk$, let $X$ be a non-singular affine variety over $\kk$, let $Y\subset X$ be an embedding over $\kk$ and let $E_\bullet^{\bullet,\bullet} $ be the associated Hodge-de Rham spectral sequence for the computation of the de Rham homology of $Y$. Then this spectral sequence is independent of the embedding beginning with the $E_2$ page, and the groups $E^{p,q}_2$ are all finite-dimensional over $\kk$. 
\end{abstract}

\section{Introduction}
In \cite{hartshorne}, Hartshorne developed algebraic de Rham homology and cohomology theories using embeddings of schemes over a field $\kk$ of characteristic zero.  The algebraic de Rham homology and cohomology are known to be independent of the embedding, and both are known to be finite-dimensional over $\kk$.

Hartshorne defined the algebraic de Rham homology groups $H^{\rm dR}_*(Y)$ of a closed subscheme $Y$ of a smooth scheme $X$ over $\kk$ as the local hypercohomology with support in $Y$ of the de Rham complex on $X$, which he establishes to be independent of $X$ and the embedding.  One reason this is useful and interesting is that even though the definition of $H^{\rm dR}_*(Y)$ is completely algebraic, makes sense for any field of characteristic 0, and not based on any notions of continuity, Hartshorne in the same paper proved a comparison theorem that says the following: if $\kk=\mathbb C$ is the field of complex numbers, then $H^{\rm dR}_i(Y)$, the $i$-th algebraic de Rham homology group of $Y$, is isomorphic to $H_i^{\rm BM}(Y_{\rm an})$, the $i$-th Borel-Moore homology group of the analytic space associated with $Y$.

Since this is computed from the hypercohomology of a complex, we may also compute it via a spectral sequence, known in this case as the Hodge-de Rham spectral sequence. The first page of this spectral sequence is given by $E_1^{n-p,n-q}=H_Y^{n-q}(X, \Omega_X^{n-p})$, where $\Omega^{\bullet}_X$ denotes the de Rham complex on $X$, and the abutment of this spectral sequence is $H^{dR}_{p+q}(Y)$.  While the abutment is known to be independent of the embedding $Y \rightarrow X$, the spectral sequence is not; in particular, the $E_1$-page is known to depend on the embedding.

In \cite{switala}, Switala showed the embedding-independence of the spectral sequence, starting with the $E_2$ page and up to degree shift, in the case where $X$ is the spectrum of a complete local ring.  Additionally, he showed that the $\kk$-spaces appearing on the $E_2$ page are finite dimensional.  Since these dimensions are independent of the embedding, these are invariants associated to the variety $Y$.  In this paper we show the corresponding results for embeddings of affine varieties into arbitrary smooth affine varieties. Our theorem is the following:

\begin{theorem}
Let $Y$ be an affine variety over a field $\kk$ of characteristic 0, and let $Y \rightarrow X$ be an embedding of $Y$ into a smooth affine variety $X$ over $\kk$.  

(a) Starting with the $E_2$ page, the Hodge-de Rham spectral sequence $$E_1^{n-p,n-q}=H^{n-q}_Y(X, \Omega_X^{n-p})\Longrightarrow H^{\rm dR}_{p+q}(Y)$$ for the computation of the algebraic de Rham homology of $Y$, associated to this embedding, is independent of the choice of $X$ and the embedding, up to a bi-degree shift.  

(b) The modules appearing on the $E_2$ page, and hence on every $E_r$ page for $r\geq 2$, are finite-dimensional $\kk$-spaces.
\end{theorem}

The precise statement of embedding-independence is as follows:  given two embeddings $Y \rightarrow X$ and $Y \rightarrow X'$ of an affine variety $Y$ into smooth affine varieties $X$ and $X'$ with $\dim X = s$ and $\dim X' = s'$, there are two associated spectral sequences $E^{\bullet, \bullet}_{\bullet}$ and ${E'}_{\bullet}^{\bullet, \bullet}$.  The theorem asserts that there is a morphism $E_{\bullet}^{\bullet, \bullet} \rightarrow {E'}^{\bullet, \bullet}_{\bullet}$ with bidegree shift $(s'-s, s'-s)$, which is an isomorphism on the $E_2$-page and later.  Precise definitions of the notions of bidegree shifts and morphisms of spectral sequences are provided in section 2 on the preliminaries.

Since the modules appearing on the $E_2$-page are finite-dimensional, the following natural question arises:

{\it Question.} Let Y be an affine variety over a field $\kk$ of characteristic 0 and let $Y \rightarrow X$ be an embedding of $Y$ into a smooth affine variety $X$ over $\kk$. Does the associated Hodge-de Rham spectral sequence for the computation of the algebraic de Rham homology of $Y$ degenerate at $E_2$?

The very recent paper \cite{arXiv:2110.08197} shows that the answer is positive for determinantal varieties, and another very recent paper \cite{arXiv:2106.04457} establishes that the answer is positive in a few more special cases. In general, however, the question is open.

The results in this paper are from the author's thesis, supervised by Gennady Lyubeznik.  We would like to thank Professor Lyubeznik both for suggesting the problem and for his guidance and support.
\section{Preliminaries}

\subsection{$\mathcal{D}$-modules}
Let $\kk$ be a field of characteristic $0$, and let $T=\kk[x_1,\dots ,x_n]$.  We write $\mathcal{D}(T,\kk)$ for the subring of End$_{\kk}$($T$) generated by $T$ and the differential operators $\delta_1,\dots ,\delta_n$, where $\delta_i=\frac{\partial}{\partial x_i}$, the usual partial differentiation with respect to $x_i$.  The ring $\mathcal{D}(T,\kk)$ is noncommutative.  We consider a filtration on $\mathcal{D}(T,\kk)$ by degree (called the Bernstein filtration), that is $\mathcal{F}_0 = k$ and $\mathcal{F}_v=\{k$-span of monomials in the variables $x_1, x_2, \dots, x_n, \delta_1, \delta_2, \dots \delta_n$ of degree $\leq v\}$ for $v > 0$.  The graded ring associated to this filtration, gr $\mathcal{D}(T,\kk)$, is commutative.  A filtration on a left $\mathcal{D}(T,\kk)$-module $M$ is called {\it good} if gr $M$ is finitely generated as a gr $\mathcal{D}(T,\kk)$-module. The module $M$ is a finitely generated left $\mathcal{D}(T,\kk)$-module, if and only if there is a good filtration of $M$ (Propositions 1.2.6 and 1.2.7 in \cite{bjork}). 

Given a finitely generated $\mathcal{D}(T,\kk)$-module $M$, we may define the dimension of $M$ as the dimension of gr $M$ as a gr $\mathcal{D}(T,\kk)$-module.  Let $\Gamma_v$ be a good filtration of $M$, and consider the Hilbert function $h(v) = \dim_{\kk}(\Gamma_v)$.  Hilbert proved that for large enough $v$, this function agrees with a polynomial in $v$.  The dimension of $M$ is defined to be the degree of this polynomial.  (More details can be found in section 1.9 of \cite{eisenbud}, where this statement appears as Theorem 1.11, for example.)

Bernstein's theorem states that for any finitely generated $\mathcal{D}(T,\kk)$-module $M$, the dimension of $M$ is bounded between $n$ and $2n$.  Those finitely generated modules of dimension $n$ form the Bernstein class of $D(T,\kk)$-modules and are called holonomic.  Holonomic modules form a full subcategory of $D(T,\kk)$-modules, closed under formation of submodules, quotient modules, and extensions. (Prop 1.5.2 in \cite{bjork})

If $M$ is a holonomic $D$-module, and $f$ is an element of $T$, then $M_f$ is a holonomic $D$-module  (Theorem 1.5.9 in \cite{bjork}). It follows by considering the Koszul complex that local cohomology modules of a holonomic module are holonomic \cite{lyubeznik}.

\subsection{K\"ahler Differentials}
Here we summarize fundamentals from \cite{eisenbud}.  Let $R$ be a ring, let $S$ be an $R$-algebra and consider the set of $R$-linear derivations of $S$: $R$-linear maps $d$ from $S$ to some $S$-module $M$ satisfying $d(ab) = ad(b) + bd(a)$.  The module of K\"{a}hler differentials of $S$ over $R$, written $\Omega_{S/R}$, is the dual object to the set of $\kk$-linear derivations of $S$.  It is the $S$-module generated by $\{ds | s \in S\}$ with relations $d(ss')=sds' + s'ds$ and $d(as + a's')=ads+a'ds'$ for any $s$ and $s'$ in $S$, and any $a, a'$ in $R$.  Observe that these two relations together require that $da=0$ for any element $a \in R$.  We can consider $d$ to be a map from $S$ to $\Omega_{S/R}$ that sends $a\in S$ to $da\in \Omega_{S/R}$, and we call $d$ the universal $R$-linear derivation.  It has the following universal property: given any $S$-module $M$ and a derviation $d': S \rightarrow M$, there exists a unique $S$-linear homomorphism $e: \Omega_{S/R} \rightarrow M$ with $e \circ d = d'$.

For a simple example, consider the case in which $R=\kk$, and $S=\kk[x_1,\dots,x_n]$.  In this situation, $\Omega_{S/R}$ is the free $S$-module on generators $dx_1, \dots dx_n$.  

We need two additional properties of the module $\Omega_{S/R}$ for our results.  First, we need the conormal sequence, which gives us a relationship between the modules of differentials for a given surjection of algebras. The following is proposition 16.3 in \cite{eisenbud}.

\begin{prop}\label{conormal}  If $\pi:T \rightarrow S$ is an surjection of $R$-algebras and $I$ is the kernel of $\pi$, then there is an exact sequence of $S$-modules $$I/I^2 \rightarrow S\otimes_T\Omega_{T/R} \rightarrow \Omega_{S/R} \rightarrow 0.$$
where the left-most map is given by the restriction of $d$ to $I$, and the middle map is given by $s\otimes dt \mapsto sd\pi(t)$
\end{prop}

We also need to know the relationship between the formation of the module of differentials and localizations, namely that they commute.  This is proposition 16.9 in \cite{eisenbud}.

\begin{prop}\label{localdiff}
Let $S$ be a $R$-algebra and let $U$ be a multiplicative subset of $S$.  Then $$\Omega_{S[U^{-1}]/R} \cong S[U^{-1}] \otimes_S \Omega_{S/R}$$
\end{prop}

\subsection{Preliminaries on de Rham cohomology}
Now let $S$ be a $\kk$-algebra, and let $\Omega^i(S) = \wedge^i\Omega_{S/\kk}$.  Then there is a complex $$0 \rightarrow S \rightarrow \Omega^1(S) \rightarrow \Omega^2(S) \rightarrow \dots$$ where the map from $\Omega^p(S)$ to $\Omega^{p+1}(S)$ is $\kk$-linear and sends $sds_1\wedge\dots\wedge ds_p$ to $ds\wedge ds_1\wedge\dots\wedge ds_p$.  This is well-defined because of properties of derivations, and is a complex because the usual wedge rules apply: $dx_i\wedge dx_i = 0$ for all $i$, and $dx_i \wedge dx_j = - dx_j\wedge dx_i$ for all $i$ and $j$.

Let $M$ be a $\mathcal{D}(T,\kk)$-module, where $T=\kk[x_1,\dots, x_n]$ as before. The de Rham complex of $M$, denoted $\Omega^\bullet(M)$, is the complex $$0 \rightarrow M \rightarrow \Omega^1(M) \rightarrow \Omega^2(M) \rightarrow \dots$$ where $\Omega^p(M) = M\otimes_T\Omega^p(T)=\bigoplus_{i_1<i_2< \cdots< i_p} M dx_{i_1} \wedge \dots \wedge dx_{i_p}$ and the map from $\Omega^p(M)$ to $\Omega^{p+1}(M)$ sends $\oplus m_{i_1,\dots, i_p} dx_{i_1} \wedge \dots \wedge dx_{i_p}$ to $\sum_{j=1}^n \oplus \delta_j(m_{i_1,\dots, i_p})dx_j\wedge dx_{i_1}\wedge \dots \wedge dx_{i_p}$, where the direct sum on the right hand side is taken over all $i_1 < i_2 \dots < i_p$ and $j$.  This sum is an element of $\Omega^{p+1}$ because the usual wedge rules apply.

The $p$-th de Rham cohomology of a $\mathcal D(T,\kk)$-module $M$, denoted $H^p(\Omega^\bullet(M))$, is the $p$-th cohomology group of $\Omega^\bullet(M)$.

We will need the following theorem, which is Theorem 6.1 in chapter 1 of \cite{bjork}, later.
\begin{theorem}\label{findim}
If $M$ is a holonomic $\mathcal{D}(T,\kk)$-module, where $T=\kk[x_1,...,x_n]$ is the ring of polynomials in variables $x_1,\dots,x_n$ over $\kk$, then $H^p(\Omega^{\bullet}(M))$ are finite dimensional vector spaces over $\kk$.
\end{theorem}

\subsection{Preliminaries on Spectral Sequences}
The following is a synopsis of necessary definitions and results from Weibel's introduction to the subject \cite{weib}.
\begin{definition}
A cohomology spectral sequence in an abelian category $\mathcal A$ is a family of objects $\{E^{p,q}_{r}\}$ defined for all integers $p,q$, and for all integers $r>a$ for some $a\geq 0$, with differential maps 
$d_r^{p,q}:E^{p,q}_{r}\rightarrow E^{p+r,q-r+1}_r$, with the property that $E_{r+1}^{p,q}$ is isomorphic to the cohomology module $\ker (d_r^{p,q}:E^{p,q}_{r}\rightarrow E^{p+r,q-r+1}_r)/\im (d_r^{p-r,q+r-1}:E^{p-r,q+r-1}_{r}\rightarrow E^{p,q}_r)$.
\end{definition}

The collection of objects in the spectral sequence of the form $E^{p,q}_{r}$ for fixed $r$, together with their associated differentials are referred to as the $E_r$-{\it page} of the spectral sequence.

The spectral sequence in this paper is an example of a {\it first-quadrant spectral sequence}, so named because the only nonzero terms are those in the first quadrant.  That is, $E^{p,q}_{r}=0$ for $p<0$ or $q<0$.  Since for any index in the spectral sequence, one can choose $r$ large enough that the differentials mapping to and from $E^{p,q}_{r}$ are maps to or from positions not in the first quadrant, such spectral sequences have the property that for any fixed $p, q$, there is some $r$ such that $E^{p,q}_{k}=E^{p,q}_{r}$ for $k>r$.  We write $E^{p,q}_{\infty}$ for the entry at that $r$.  
 
A spectral sequence is said to converge to its abutment $H^*$ if $E^{p,q}_{\infty}$ give the quotients of a finite filtration of $H^{*}$.  

Double complexes are a common source of spectral sequences. Namely, let $C^{\bullet, \bullet}$ be a first quadrant double complex (i.e. $C^{p,q}=0$ for $p<0$, or $q<0$) with vertical differentials $d_{ver}^{p,q}:C^{p,q}\to C^{p,q+1}$ and horizontal differentials $d_{hor}^p:C^{p,q}\to C^{p+1,q}$. This yields a spectral sequence with $E_0^{p,q}=C^{p,q}$ and $d_0^{p,q}=d_{ver}^{p,q}$ where $d_{ver}^{p,q}:C^{p,q}\to C^{p, q+1}$ is the vertical differential of the double complex, while on the $E_1$ page, $E_1^{p,q}=H^q(C^{p,\bullet})$, the $q$-th cohomology of the $p$-th column, and $d_1:E_1^{p,q}\to E_1^{p+1,q}$ is induced by the horizontal differential in the double complex (this is known as the column-filtered spectral sequence of $C^{\bullet,\bullet}$; there also is a row-filtered spectral sequence which we will not need). This spectral sequence abuts to $H^\bullet({\rm Tot}(C^{\bullet,\bullet}))$, the cohomology of the total complex of $C^{\bullet, \bullet}$, in the following sense. Let $C^{\bullet,\bullet} _{\geq t}$ be $C^{\bullet,\bullet}$ truncated at column $t$, i.e. the double subcomplex of $C^{\bullet,\bullet}$ defined by $C^{p,q}_{\geq t}=C^{p,q}$ if $p\geq t$ and $C^{p,q}_{\geq t}=0$ if $p<t$. Then $C^{\bullet,\bullet}_{\geq t+1}$ ia a double subcomplex of $C^{\bullet,\bullet}_{\geq t}$ and there is a filtration $\dots\supset F^{t}H^\bullet\supset F^{t+1}H^\bullet\supset\dots$ on $H^\bullet({\rm Tot}(C^{\bullet,\bullet}))$, namely, $F^tH^\bullet$ is the image of $H^\bullet({\rm Tot}(C^{\bullet,\bullet}_{\geq t}))$ under the natural map induced by the natural inclusion Tot$(C^{\bullet,\bullet}_{\geq t})\subset {\rm Tot}(C^{\bullet,\bullet})$ and $F^{t}H^n/F^{t+1}H^n\cong E_\infty^{t,n-t}$.

Let $F:\mathcal A\to \mathcal B$ be a left-exact additive functor from an abelian category $\mathcal A$ to an abelian category $\mathcal B$ and let $K^\bullet$ be a complex in $\mathcal A$ concentrated in degrees $p\geq 0$. The hyperderived functors $R^qF(K^\bullet)$ are defined as $H^q({\rm Tot} (F(L^{\bullet,\bullet})))$, the $q$-th cohomology of the total complex of $F(L^{\bullet,\bullet})$ where $L^{\bullet,\bullet}$ is a Cartan-Eilenberg resolution of $K^\bullet$; these hyperderived functors are independent of the particular Cartan-Eilenberg resolution. The spectral sequence of the double complex $F(L^{\bullet,\bullet})$ described in the preceding paragraph also is independent of the particular Cartan-Eilenberg resolution (starting from the $E_1$ page) and it abuts to $R^\bullet F(K^\bullet)$. The $E_1^{p,q}$ entry is $E_1^{p,q}=H^q(F(L^{p,\bullet}))=R^qF(K^{p})$, the $q$-th derived functor of $F$ evaluated at $K^p$.

A morphism of spectral sequences $f:\{E^{\bullet,\bullet}_\bullet\} \rightarrow \{{E'}^{\bullet,\bullet}_\bullet\}$ consists of homomorphisms $f^r:E^{\bullet, \bullet}_r \rightarrow E'^{\bullet,\bullet}_r$, that commute with the differentials in the spectral sequence.  

\begin{lemma}[Lemma 5.2.4 in \cite{weib}]\label{mapping} Let $f:\{E^{p,q}_r\} \rightarrow \{E'^{p,q}_r\}$ be a morphism of spectral sequences such that $f^{r_0}:E^{p,q}_{r_0} \rightarrow E'^{p,q}_{r_0}$ is an isomorphism for some $r_0$, then $f^r$ is an isomorphism for all $r>r_0$.  In particular $f^{\infty}: E^{p,q}_{\infty} \rightarrow E'^{p,q}_{\infty}$ is an isomorphism.
\end{lemma}

We will also need the following Lemma, which is Lemma 2.13 in \cite{switala}:

\begin{lemma}\label{canusesequence}
Let $\mathcal A$ be an Abelian category with enough injectives, $\mathcal B$ another Abelian category, and $F:\mathcal A \rightarrow \mathcal B$ a left-exact additive functor.  Let $K^\bullet$ be a complex in $\mathcal A$ that is concentrated in degrees $p \geq 0$ and let $L^{\bullet, \bullet}$ be a double complex in $\mathcal A$ whose objects $L^{p,q}$ are $F$-acyclic.  Suppose additionally that for each $p$, $L^{p, \bullet}$ is a resolution of $K^p$.  Then the spectral sequence for the hyperderived functors of $F$ applied to $K^\bullet$, which starts $E^{p,q}_1 = R^qF(K^p)$ and abuts to ${\bf R}^{p+q}F(K^{\bullet})$ is isomorphic to the column-filtered spectral sequences of the double complex $F(L^{\bullet, \bullet})$
\end{lemma}

\begin{lemma}[see page 30 in \cite{grothendieck}]\label{dbtoss}
A morphism of double complexes induces a morphism of corresponding spectral sequences.
\end{lemma}

\begin{lemma}\label{cptoss}
A morphism of complexes induces a morphism of corresponding spectral sequences for a hyperderived functor.
\end{lemma}

\begin{proof}
Let $f:C^{\bullet} \rightarrow C'^{\bullet}$ be a morphism of complexes.  Let $\mathcal J^{\bullet,\bullet}$ and $\mathcal J'^{\bullet, \bullet}$ be Cartan-Eilenberg resolutions of $C^{\bullet}$ and $C'^{\bullet}$, respectively.  Then $f$ induces a map $\mathcal J^{\bullet,\bullet} \rightarrow \mathcal J'^{\bullet, \bullet}$ that is unique up to homotopy (see page 33 in \cite{grothendieck}), which induces a map $F(\mathcal J^{\bullet,\bullet}) \rightarrow F(\mathcal J'^{\bullet, \bullet})$ that is unique up to homotopy.  Thus from \autoref{dbtoss} this induces a morphism of spectral sequences for the hyperderived functors of a functor $F$ applied to $C^{\bullet}$ and $C'^{\bullet}$.  Chain homotopic maps induce the same morphisms on cohomology, so this is well-defined.
\end{proof}

We use \autoref{canusesequence} to construct the Hodge-de Rham spectral sequences for our embeddings.  Given a de Rham complex, we take injective resolutions to obtain a double complex, and use its column-filtered spectral sequence.

Finally we recall the notion of degree-shifted morphisms of spectral sequences.

\begin{definition}
[see Definition 2.1 in \cite{switala}] Let $E$ and $E'$ be two spectral sequences with respective differentials $d$ and $d'$. If $a,b\in \mathbb Z$, a morphism $u:E\to E'$ of bidegree $(a,b)$ is a family of morphisms $u^{p.q}_r:E^{p,q}_r\to E'^{p+a,q+b}_r$ such that the $u_r^{p,q}$ are compatible with the differentials and $u^{p,q}_{r+1}$ is induced on cohomology by $u_r^{p,q}$.
\end{definition}

As is pointed out in \cite{switala} right after Definition 2.1, a degree-shifted analogue of \autoref{mapping} also holds. It is not hard to see that degree-shifted versions of \autoref{dbtoss} and \autoref{cptoss} hold as well (that is, a morphism of double complexes of bidegree $(a,b)$ induces a morphism of corresponding spectral sequences of the same degree, and a morphism of complexes $f:C^\bullet\to \mathcal C^\bullet$ of degree $n$, i.e. $f$ sends $C^t$ to $\mathcal C^{t+n}$ for every $t$,  induces a morphism of corresponding spectral sequences of bidegree $(0,n)$).

\section{A morphism of spectral sequences}
The main result of this section is the following.
\begin{prop}\label{morphism}
Let $Y$ be an affine variety, and let $Y \rightarrow X$ be an embedding of $Y$ into a smooth affine variety $X$.  Let $X \rightarrow \mathbb{A}^n$ be an embedding of $X$ into affine space.
There is a morphism of spectral sequences between the Hodge-de Rham local hypercohomology spectral sequence associated to the embedding $Y\rightarrow X$, and the Hodge-de Rham local hypercohomology spectral sequence associated to the embedding $Y \rightarrow \mathbb{A}^n$ given by the composition $Y \rightarrow X \rightarrow \mathbb {A}^n$.  This morphism is of bidegree shift $(n-s, n-s)$, where $s=\dim X$.
\end{prop}

The proof results from a series of lemmas.  Throughout this section, let $T$ be the ring $\kk[x_1,...,x_n]$, the coordinate ring of $\mathbb A^n$.  Let $Y=V(I)$ be an affine variety with an embedding to a smooth affine variety $X=V(J)$ of dimension $s<n$, itself embedded into $\Spec T$, where $J \subset I \subset T$. We write $R$ for $T/J$ and $S$ for $T/I$. 

\begin{lemma}
Let $\hat{T}$ be the $J$-adic competion of $T$.  The natural map $T \rightarrow \hat{T}$ induces a map of de Rham complexes $\Omega^{\bullet}(T) \rightarrow \hat{T} \otimes_T \Omega^{\bullet}(T)$, given by $a \mapsto 1 \otimes a$.  According to \autoref{cptoss}, this map of complexes induces a morphism of spectral sequences corresponding to local hypercohomology of the de Rham complex of $T$ with support in $I$ and the local hypercohomology of the complex $\hat{T}\otimes_T \Omega^{\bullet}(T)$ with support in $I$.  This morphism of spectral sequences is an isomorphism starting on the $E_1$ page and is degree-preserving, i.e it is of bidegree shift (0,0).
\end{lemma}

\begin{proof}
Write $E_\bullet^{\bullet,\bullet}$ for the spectral sequence corresponding to local hypercohomology of the de
 Rham complex of $T$, and ${\bf E}_\bullet^{\bullet,\bullet}$ for the spectral sequence corresponding to local hypercohomology of the complex $\hat{T}\otimes_T \Omega^{\bullet}(T)$. The $E_1$ pages of these spectral
sequences consist of the modules $E_1^{p,q} = H_I^{q}(T \otimes_T \Omega^p(T)) = H^q_I(T) \otimes_T \Omega^p(T)$ and ${\bf E}_1^{p,q}= H^q_I(\hat{T} \otimes_T \Omega^p(T)) = H^q_I(\hat{T}) \otimes_T  \Omega^p(T)$.  The map $E_1^{p,q} \rightarrow {\bf E}_1^{p,q}$, is induced by the map $H_I^q(T) \rightarrow H^q_I(\hat{T})$, which in turn is induced by the natural map $T \rightarrow \hat{T}$.  This map $H_I^q(T) \rightarrow H^q_I(\hat{T})=H^q_I(T)\otimes_T\hat{T}$ is an isomorphism since $H^q_I(T)$ is $J$-torsion and $\hat{T}$ is $J$-adically complete.  Thus $E_1^{p,q} \rightarrow {\bf E}_1^{p,q}$ is an 
  isomorphism, so from \autoref{mapping} the map $E_\bullet^{\bullet,\bullet} \rightarrow {\bf E}_\bullet^{\bullet,\bullet}$ is an isomorphism, starting on the $E_1$ page.
\end{proof}

Since $\hat{T}$ is naturally a $D(T,\kk)$-module, the complex $\hat{T}\otimes_T\Omega^{\bullet}(T)$ is the de Rham complex of $\hat{T}$.

\begin{lemma}
There exists an morphism of spectral sequences $E_\bullet^{\bullet, \bullet} \rightarrow \tilde{\bf E}_\bullet^{\bullet, \bullet}$ of bidegree shift $(n-s,0)$, where $E_\bullet^{\bullet, \bullet}$ is the local hypercohomology spectral sequence with support in $I$ associated to de Rham complex of $\hat{T}$, and $\tilde {\bf E}_\bullet^{\bullet, \bullet}$ is the local hypercohomology spectral sequence with support in $I$ associated to the de Rham complex of $H^{n-s}_J(\hat{T})$.
\end{lemma}

\begin{proof} 
Let $\mathcal J^{\bullet}$ be a minimal injective resolution of $\hat{T}$.  Since $\hat{T}$ is Gorenstein, this is the Cousin complex $C^\bullet$, which is defined by $C^{-2}(\hat{T})=0$, $C^{-1}(\hat{T})=\hat{T}$, and when $i\geq 0$, $C^i(\hat{T})=\oplus(\coker d^{i-2})_{\p}$, where the direct sum extends over all $\p$ of height $i$, and the differentials are the canonical localization maps.  These modules are $D$-modules, and the maps are $D$-linear maps (see Lemma 2.25 in \cite{switala}).

Consider the double complex $\mathcal J^{\bullet, \bullet}$ where $\mathcal J^{p,q}=\mathcal J^q\otimes_{\hat{T}} \Omega^p(\hat{T})$.  The columns of this double complex are the minimal injective resolutions of $\Omega^t$ and the rows are the de Rham complexes of $\mathcal J^t$. Now in the complex $\Gamma_J(\mathcal J)^{\bullet, \bullet}$ given by applying the functor $\Gamma_J(-)$ to each entry, the $p$th row, where $p=n-s$, is the first nonzero row, since ht$(J)$ is $n-s$.  The vertical cohomology in degree $n-s$ of this complex is $H^{n-s}_J(\Omega^{\bullet}(\hat{T}))$ and the vertical cohomology in degrees $>n-s$ is zero, since $J$ is locally a complete intersection of height $n-s$. So by performing a row shift of $s-n$ on the double complex $\Gamma_J(\mathcal J)^{\bullet, \bullet}$, we obtain the double complex for the de Rham complex of $H^{n-s}_J(\hat{T})$, i.e. the columns are injective resolutions of $H^{n-s}_J(\hat{T})\otimes_T\Omega^t$ and the rows are the de Rham complexes of $\mathcal J^t$.  Write $\mathcal J_0^{\bullet, \bullet}$ for that complex.

Now, since $\Gamma_I(\Gamma_J(\mathcal J^{p,q}))=\Gamma_I(\mathcal J^{p,q})$, by applying $\Gamma_I$ to each complex we may identify $\Gamma_I(\mathcal J^{\bullet, \bullet})$ with $\Gamma_I(\mathcal J_0^{\bullet,\bullet})$, by a row shift of $n-s$ on $\Gamma_I(\mathcal J^{\bullet, \bullet})$.  This isomorphism induces a morphism on spectral sequences by \autoref{cptoss}
\end{proof}

We now construct a morphism between the de Rham complex of $X$ and the de Rham complex of $H_J^{n-s}(\hat{T})$. The morphism will be shifted in degrees by $n-s$, i.e.~the degree $i$ piece of the de Rham complex of $X$ will map to the degree $i+n-s$ piece of the de Rham complex of  $H_J^{n-s}(\hat{T})$. This will produce the other $n-s$ half of the bidegree shift in the isomorphism of spectral sequences. For this we will need a few more lemmas.

Recall that a $\kk$-algebra $S$ is {\it formally smooth} over $\kk$ if given any $\kk$-algebra $A$ and ideal $\mathfrak A\subset A$ of square zero, every $\kk$-algebra map $S\to A/\mathfrak A$ lifts to a $\kk$-algebra map $S\to A$ (See Definition 2.1 of chapter 17 of \cite{cring}). The basic example of a formally smooth $\kk$-algebra is the polynomial ring $\kk[x_1,\dots,x_n]$ (Example 2.2 in chapter 17 of \cite{cring}). We quote the following result of Grothendieck.

\begin{theorem}
[\cite{cring} Ch.~17, 2.4] Suppose $T$ is a formally smooth $\kk$-algebra and let $J\subset T$ be an ideal. Then $R=T/J$ is a formally smooth $\kk$-algebra if and only if the conormal sequence $$J/J^2\to \Omega_{T/\kk} \otimes_TR\to\Omega_{R/\kk}\to 0$$ is split exact.
\end{theorem}

We use this result in the proof of the following lemma.

\begin{lemma}\label{khom}
There exists a $\kk$-algebra homomorphism $\phi:R=\hat T/J\to \hat T$ such that the composition $\psi\circ\phi:R\to R$ is the identity, where the homomorphism $\psi:\hat T\to \hat T/J\cong R$ is the natural surjection.
\end{lemma}

\emph{Proof.} Since in our case the polynomial ring $T=\kk[x_1,\dots,x_n]$ is indeed formally smooth over $\kk$ and the conormal sequence is indeed split exact (since $R=T/J$ is locally a complete intersection, see Ex 17.12 in \cite{eisenbud}), we conclude that $R$ is formally smooth over $\kk$. Thus there exists a lifting $R\to \hat T$ (Remark on pages 199-200 in \cite{matsumura}). \qed

From now on we regard $\hat T$ as an $R$-algebra via $\phi$.  We will denote the module of K\"ahler differentials of $T$ as an $R$-algebra as $\Omega_{T/R}$.  Next we quote the following result.

\begin{lemma}
[Lemma 10.130.11 in \cite{cring}] Let $R\to S$ be a ring map. Let $I\subset S$ be an ideal. Let $n\geq 1$ be an integer. Set $S'=S/I^{n+1}$. The map $\Omega_{S/R}\to \Omega_{S'/R}$ induces an isomorphism $$\Omega_{S/R}\otimes_SS/I^n\to\Omega_{S'/R}\otimes_{S'}S/I^n.$$
\end{lemma}

We use this result in the proof of the following lemma.

\begin{lemma}\label{tmap}
The $R$-algebra structure on $\hat T$ induces a natural injective $\hat T$-module map $$\hat T\otimes_R\Omega_{R/\kk}\to \hat T\otimes_T\Omega_{T/\kk}$$ and a direct sum decomposition $$\hat T\otimes_T\Omega_{T/\kk}=(\hat T\otimes_R\Omega_{R/\kk})\oplus P$$ where $P$ is a projective $\hat T$-module such that $P/JP\cong J/J^2$.
\end{lemma}

\emph{Proof.} Let $R\subset T/J^{n+1}$ be the image of $R=\phi(R)\subset \hat T$ under the natural surjection $\hat T\to T/J^{n+1}$. The ring maps $\kk\to R\to T/J^{n+1}$ induce the standard exact sequence of $T/J^{n+1}$-modules $$T/J^{n+1}\otimes_R\Omega_{R/\kk}\to \Omega_{(T/J^{n+1})/\kk}\to \Omega_{(T/J^{n+1})/R}\to 0.$$
Tensoring this exact sequence with $T/J^n\otimes_{T/J^{n+1}}$ and considering that the tensor product is right exact we get the exact sequence $$T/J^{n}\otimes_R\Omega_{R/\kk}\to T/J^n\otimes_{T/J^{n+1}}\Omega_{(T/J^{n+1})/\kk}\to T/J^n\otimes_{T/J^{n+1}}\Omega_{(T/J^{n+1})/R}\to 0.$$ According to the preceding lemma, the term in the middle is isomorphic to $T/J^n\otimes_T\Omega_{T/\kk}$ while the term on the right is isomorphic to $T/J^n\otimes_T\Omega_{\hat{T}/R}$ (considering that $\hat T/J^{n+1}\cong T/J^{n+1}).$ Hence the exact sequence becomes $$T/J^{n}\otimes_R\Omega_{R/\kk}\to T/J^n\otimes_T\Omega_{T/\kk}\to  T/J^n\otimes_T\Omega_{\hat{T}/R}\to 0.$$

We claim that the term on the right is a projective, i.e.~locally free, $T/J^n$-module of rank $n-s$, where $s$ is the dimension of $R$. It is sufficient to prove this locally, so let $f\in R\subset \hat T$ be an element such that $J_f$ is a complete intersection, generated by $n-s$ elements $z_1,\dots, z_{n-s}$. In this case $(T/J^{n+1})_f=R_f[z_1,\dots,z_{n-s}]/(z_1,\dots,z_{n-s})^{n+1}$ where $R_f[z_1,\dots,z_{n-s}]$ is the polynomial ring in $n-s$ variables $z_1,\dots,z_{n-s}$ over $R_f$. Since the module $\Omega_{(R_f[z_1,\dots,z_{n-s}])/R_f}$  is the free $R_f[z_1,\dots,z_{n-s}]$-module of rank $n-s$, we conclude from the preceding lemma, that $(T/J^n\otimes_T\Omega_{(T/J^{n+1})/R})_f=(T/J^n)_f\otimes_T\Omega_{(R_f[z_1,\dots,z_{n-s}])/R_f}$ is the free $(T/J^n)_f$-module of rank $n-s$. This proves the claim.

Since $\Omega_{R/\kk}$ is a projective, i.e.~locally free $R$-module of rank $s$, the module $(T/J^n)\otimes_R\Omega_{R/\kk}$ is a locally free, i.e.~projective $T/J^n$-module of rank $s$. Since the module in the middle is a free module of rank $n$ and the module on the right is a projective $T/J^n$-module of rank $n-s$, i.e.~the rank in the middle is the sum of the two ranks at the ends, we conclude that the map on the left is injective. Since this is a short exact sequence of projective $T/J^n$-modules, it splits, i.e.~we get a direct sum decomposition $$T/J^n\otimes_T\Omega_{T/\kk}=(T/J^n\otimes_R\Omega_{R/\kk})\oplus (T/J^n\otimes_T\Omega_{\hat T/R}).$$

As $n$ runs through all positive integers these direct sum decompositions form an inverse systems in which all maps are surjective. Taking the inverse limits we get the direct sum decomposition of $J$-adic completions $$\hat \Omega_{T/\kk}=(\hat T\otimes_R\Omega_{R/\kk})\oplus\hat\Omega_{\hat T/R}.$$ Since $\Omega_{T/\kk}$ is a finitely generated $T$-module, $\hat \Omega_{T/\kk}=\hat T\otimes_T\Omega_{T/\kk}$. The position of $\hat T\otimes_R\Omega_{R/\kk}$ as a submodule in $\hat T\otimes_T\Omega_{T/\kk}$ is induced by the maps on the left hand sides in the above exact sequences which are in turn induced by the $R$-algebra structure of $\hat T$. This proves the first statement of the lemma.

$\Omega_{\hat T/R}$ is not a finitely generated $\hat T$-module; yet its $J$-adic completion $\hat \Omega_{\hat T/R}$ is the inverse limit of projective $T/J^n$-modules $T/J^n\otimes_T\Omega_{\hat T/R}$ of ranks $n-s$ and all maps in this inverse system are surjective, hence $\hat \Omega_{\hat T/R}$ is a projective $\hat T$-module of rank $n-s$. Setting $P=\hat \Omega_{\hat{T}/R}$, this proves the direct sum decomposition statement of the lemma. 

Note that $P/JP=\hat T/J\otimes_{\hat{T}} \hat \Omega_{\hat T/R}=T/J\otimes_T\Omega_{\hat T/R}$, since all completions are the $J$-adic completions.  The conormal sequence $$J/J^2 \rightarrow T/J\otimes\Omega_{\hat T/R} \rightarrow \Omega_{R/R} \rightarrow 0$$ is split exact, from 10.131.10 in \cite{stacks}, since the surjection $\hat{T} \rightarrow T/J$ has a right inverse by \autoref{khom}. Since $\Omega_{R/R}=0$, this implies that $J/J^2 \rightarrow  T/J\otimes \Omega_{\hat T/R}$ is an isomorphism, and thus $P/JP \cong J/J^2$. \qed

Setting $\hat T\otimes_R\Omega_{R/\kk}=Q$ we get a direct sum decomposition  $\hat T\otimes_T\Omega_{T/\kk}=Q\oplus P$. Thus if $i\geq n-s$, then $\wedge^{n-s}P\otimes_T\wedge^{i-(n-s)}Q$ is a submodule of $\wedge^i(\hat T\otimes_T\Omega_{T/\kk})=\hat T\otimes_T\Omega^i(T)$ (in fact, $\wedge^{n-s}P\otimes_T\wedge^{i-(n-s)}Q$ is a direct summand of $\wedge^i(\hat T\otimes_T\Omega_{T/\kk})=\hat T\otimes_T\Omega^i(T)$). 

\begin{lemma}\label{differential}
The differential in the de Rham complex $\hat T\otimes_T\Omega_{T/\kk}$ sends the submodule $\wedge^{n-s}P\otimes_T\wedge^{i-(n-s)}Q$ of $\hat T\otimes_T\Omega^i(T)$ to the submodule $\wedge^{n-s}P\otimes_T\wedge^{i-(n-s)+1}Q$ of $\hat T\otimes_T\Omega^{i+1}(T)$. Thus there is a subcomplex $$0\to \wedge^{n-s}P\to \wedge^{n-s}P\otimes_T Q\to \wedge^{n-s}P\otimes_T\wedge^2Q\to\dots\to \wedge^{n-s}P\otimes_T\wedge^sQ\to 0$$ of the de Rham complex $\hat T\otimes_T\Omega_{R/\kk}$ shifted in degrees by $n-s$ (i.e.~the degree $i$ piece of the subcomplex sits in the degree $i+(n-s)$ piece of the de Rham complex).
\end{lemma}

\emph{Proof.} Since the de Rham complex commutes with localization, it is enough to prove this statement after localization at every maximal ideal $\mathfrak m$ of $\hat T$. Since completion of a local ring at the maximal ideal is faithfully flat, for every submodule $L\subset N$ of a $\hat T$-module $N$, we have that $\tilde L\cap N=L$, where $\tilde L$ is the $\mathfrak m$-adic completion of $L$. Hence it is enough to prove the statement after completion with respect to $\mathfrak m$. 

Let $x_1,\dots, x_s\in R$ generate the maximal ideal $\mathfrak m\cap R$ of $R$ and let $y_1,\dots, y_{n-s}\in J$ generate $J_{\mathfrak m}$. The $\mathfrak m\cap R$-adic completion of $R$ is the formal power series ring $R'=\kk[[x_1,\dots, x_s]]$ and the $\mathfrak m$-adic completion of $\hat T$ is the formal power series ring $S=R'[[y_1,\dots, y_{n-s}]]$.  The $\mathfrak m$-adic completion of $\Omega_{T/\kk}$, i.e.~$S\otimes_T\Omega_{T/\kk}$, is the free $S$-module on $dx_1,\dots, dx_s,dy_1,\dots, dy_{n-s}$ and the submodule $S\otimes_R\Omega_{R/\kk}$ of $S\otimes_T\Omega_{T/\kk}$, whose existence is claimed in \autoref{tmap}, is the free $S$-submodule generated by $dx_1,\dots,dx_s$. 

Set $\tilde y=dy_1\wedge\dots\wedge dy_{n-s}$. The module $S\otimes_T(\wedge^{n-s}P\otimes_T\wedge^{i-(n-s)}Q)$ is the free $S$-module on all elements $dx_{j_1}\wedge\dots \wedge dx_{j_{i-(n-s)}}\wedge\tilde y$ as $\{j_1,\dots,j_{i-(n-s)}\}$ run through all $(i-(n-s))$-element subsets of $\{1,\dots s\}$, for every $i$. The de Rham differential clearly sends every $dx_{j_1}\wedge\dots \wedge dx_{j_{i-(n-s)}}\wedge\tilde y$ to the $S$-submodule of $S\otimes_T(\wedge^{n-s}P\otimes_T\wedge^{i-(n-s)+1}Q)$ spanned by all elements $dx_{j_1}\wedge\dots \wedge dx_{j_{i-(n-s)+1}}\wedge\tilde y$. \qed

\begin{lemma}\label{annihilator} 
Let $M$ be the submodule of $H^{n-s}_J(\hat T)\otimes_T\wedge^{n-s}P$ annihilated by $J$. Then $M\cong T/J=R$.
\end{lemma}

\begin{proof}The composition of functors ${\rm Hom}_{\hat T}(\hat T/J,-)={\rm Hom}_{\hat T}(\hat T/J,\Gamma_J(-))$ leads to the spectral sequence $$E_2^{p,q} = {\rm Ext}^p_{\hat T}(\hat T/J,H^q_J(\hat T))\Longrightarrow {\rm Ext}^{p+q}_{\hat T}(\hat T/J,\hat T).$$ Since $\hat T/J$ is regular, $H^q_J(\hat T)\ne 0$ if and only if $n-s=q$. Thus the spectral sequence yields an isomorphism $${\rm Hom}_{\hat T}(\hat T/J,H^{n-s}_J(\hat T))\cong {\rm Ext}^{n-s}_{\hat T}(\hat T/J,\hat T).$$

Recall that $P/JP\cong J/J^2$. Thus there is a surjection $P\to J/J^2$. Since $P$ is projective, this surjection lifts to a $\hat T$-linear
map $\phi:P\to J$. Let $J'\subset J$ be the image of $\phi$. By Nakayama's lemma, $J'_\mathfrak m=J_\mathfrak m$ for every maximal ideal $\mathfrak m$ that contains $J$. We claim that every maximal ideal of $\hat T$ contains $J$. Indeed, assume $\mathfrak m$ is a maximal ideal that does not contain $J$. Then $\mathfrak m+J=\hat T$, i.e.~there exist elements $y\in J$ and $x\in \mathfrak m$ with $x+y=1$. But then $x=1-y$ is invertible in $\hat T$, for $(1-y)^{-1}=1+y+y^2+y^3+\dots\in \hat T$. Therefore $x$ cannot be contained in any maximal ideal. This contradiction proves the claim. Thus $J'_\mathfrak m=J_\mathfrak m$ for every maximal ideal $\mathfrak m$ of $\hat T$, i.e.~$J'=J$. In other words, the map $\phi:P\to J$ is surjective.

The map $\phi:P\to J$ can be completed to a Koszul resolution $\wedge^\bullet P$ of $\hat T/J$ as follows:
$$0\to \wedge^{n-s}P\to \wedge^{n-s-1}P\to \dots\to\wedge ^2P\to \wedge P(=P)\to \wedge^0P(=\hat T)\to0$$
where the $\hat T$-linear differential $d_t:\wedge^tP\to\wedge^{t-1}P$ is given by $$d_t(x_1\wedge\dots\wedge x_t)=\Sigma_j(-1)^j\phi(x_j)(x_1\wedge\dots\wedge x_{j-1}\wedge x_{j+1}\dots\wedge x_t).$$
This complex is a resolution of $\hat T/J$ in the sense that its 0-th homology is $\hat T/J$ while all other homology groups vanish.

The module Ext$^{n-s}_{\hat T}(\hat T/J,\hat T)$ is the $(n-s)$th homology of the complex Hom$_{\hat T}(\wedge^\bullet P,\hat T)$, i.e.~it is the cokernel of the map $${\rm Hom}_{\hat T}(d_{n-s},\hat T):{\rm Hom}_{\hat T}(\wedge^{n-s-1}P, \hat T)\to {\rm Hom}_{\hat T}(\wedge^{n-s}P,\hat T).$$
The image of this map is $J{\rm Hom}_{\hat T}(\wedge^{n-s}P,\hat T)$, i.e.~$${\rm Ext}^{n-s}_{\hat T}(\hat T/J,\hat T)\cong {\rm Hom}_{\hat T}(\wedge^{n-s}P,\hat T)/J{\rm Hom}_{\hat T}(\wedge^{n-s}P,\hat T)\cong {\rm Hom}_{T/J}(\wedge^{n-s}P/J\wedge^{n-s}P, T/J).$$
Since $\wedge^{n-s}P/(J\wedge^{n-s}P)\cong \wedge^{n-s}(P/JP)=\wedge^{n-s}(J/J^2)$, we conclude that $${\rm Ext}^{n-s}_{\hat T}(\hat T/J,\hat T)\cong {\rm Hom}_{T/J}(\wedge^{n-s}(J/J^2),T/J).$$
Thus the module ${\rm Hom}_{\hat T}(\hat T/J,H^{n-s}_J(\hat T))$, i.e.~the submodule of $H^{n-s}_J(\hat T)$ annihilated by $J$, is isomorphic to ${\rm Hom}_{T/J}(\wedge^{n-s}(J/J^2),T/J).$

Let $M={\rm Hom}_{\hat T}(\hat T/J, H^{n-s}_J(\hat T)\otimes_T\wedge^{n-s}P)$, i.e.~$M$ is the submodule of $H^{n-s}_J(\hat T)\otimes_T\wedge^{n-s}P$ annihilated by $J$. Since the module $\wedge^{n-s}P$ is flat over $T$, we conclude that $M\cong {\rm Hom}_{\hat T}(\hat T/J, H^{n-s}_J(\hat T))\otimes_T\wedge^{n-s}P$. This, as has just been proven, is isomorphic to ${\rm Hom}_{T/J}(\wedge^{n-s}(J/J^2),T/J)\otimes_T\wedge^{n-s}P$. Since the module ${\rm Hom}_{T/J}(\wedge^{n-s}(J/J^2),T/J)$ is annihilated by $J$, this tensor product is isomorphic to ${\rm Hom}_{T/J}(\wedge^{n-s}(J/J^2),T/J)\otimes_{T/J}(\wedge^{n-s}P/J\wedge^{n-s}P)$ which (since $\wedge^{n-s}P/J\wedge^{n-s}P=\wedge^{n-s}(P/JP=\wedge^{n-s}(J/J^2)$ is isomorphic to ${\rm Hom}_{T/J}(\wedge^{n-s}(J/J^2),T/J)\otimes_{T/J}\otimes\wedge^{n-s}(J/J^2)$. This last module is isomorphic to $T/J$ since $\wedge^{n-s}(J/J^2)$ is an invertible $T/J$-module and ${\rm Hom}_{T/J}(\wedge^{n-s}(J/J^2),T/J)$ is the inverse of this module in the Picard group of $T/J$.
\end{proof}

At last we are ready to prove \autoref{morphism}.

\emph{Proof of \autoref{morphism}.} It is enough to prove that the de Rham complex of $R$ is a subcomplex of the de Rham complex of $H^{n-s}_J(\hat T)$, shifted in degrees by $n-s$ (i.e.~the degree $i$ piece of the de Rham complex of $R$ sits in the degree $i+n-s$ piece of the de Rham complex of $H^{n-s}_J(\hat T)$). 

Denote by $\wedge^{n-s}P\otimes_T\wedge^\bullet Q$ the subcomplex of $\hat T\otimes_T\Omega^\bullet(T)$ whose existence was proved in \autoref{differential}. Tensoring with $H^{n-s}_J(T)$ we see that $H^{n-s}_J(T)\otimes_T\wedge^{n-s}P\otimes_T\wedge^\bullet Q$ is a subcomplex of $H^{n-s}_J(T) \otimes_T \Omega^\bullet(T)$ shifted in degrees by $n-s$ (tensoring this subcomplex results in a subcomplex because in every degree $i$ the submodule $\wedge^{n-s}P\otimes_T\wedge^iQ$ is a direct summand of $\hat T\otimes_T\Omega^{i+n-s}(T)$). 

We claim that the differentials in the complex $H^{n-s}_J(T) \otimes_T \wedge^{n-s}P\otimes_T\wedge^\bullet Q$ send elements annihilated by $J$ to elements annihilated by $J$ (this is certainly not true for the differentials in the ambient complex $H^{n-s}_J(T) \otimes_T \Omega^\bullet(T))$, i.e.~${\rm Hom}_T(T, H^{n-s}_J(T) \otimes_T \wedge^{n-s}P\otimes_T\wedge^\bullet Q)$ is a complex. It is sufficient to show this after completion at every maximal ideal of $\hat T$, in which case (using notation from the proof of \autoref{differential}) $H^{n-s}_J(\hat T)$ is the free $R'$-module on the monomials $y_1^{i_1}\cdots y_{{n-s}}^{i_{n-s}}$ where $i_1,\dots, i_{n-s}$ run through all negative integers $\leq -1$ and the annihilator of $J$ in $H^{n-s}_J(\hat T)$ is the $R'$-submodule spanned by $(y_1\cdots y_{n-s})^{-1}$. The differentials in the complex $H^{n-s}_J(T) \otimes_T \wedge^{n-s}P\otimes_T\wedge^\bullet Q$ involve differentiations only with respect to $x_1,\dots, x_s$ and leave the $y_1,\dots, y_{n-s}$ alone. Thus they send the submodule spanned by $(y_1\cdots y_{n-s})^{-1}$ to the submodule spanned by $(y_1\cdots y_{n-s})^{-1}$. This proves the claim. Thus ${\rm Hom}_T(T, H^{n-s}_J(T) \otimes_T \wedge^{n-s}P\otimes_T\wedge^\bullet Q)$ is a complex, in fact a subcomplex of $H^{n-s}_J(T) \otimes_T \wedge^{n-s}P\otimes_T\wedge^\bullet Q$.

\autoref{annihilator} shows that the degree 0 piece of this complex is isomorphic to $R=T/J$, while the formula $\hat{T}/J \otimes_{\hat{T}} Q = (\hat{T}/J) \otimes_{\hat{T}} \hat{T} \otimes_R \Omega_{R/\kk}= \hat{T}/J \otimes_R \Omega_{R/\kk} = \Omega_{R/\kk}$ shows that the degree 1 piece is isomorphic to $\Omega_{R/\kk}$. Thus one sees (again by going to the completion at every maximal ideal and considering the explicit formulas for the differentials in the subcomplex) that the differentials in this subcomplex are the differentials in $\Omega^\bullet(R)$, the de Rham complex of $R$. This completes the proof of \autoref{morphism}. \qed

\section{The isomorphism on the $E_2$ page}

In this section we show the morphism of spectral sequences found in the previous section is an isomorphism starting at the $E_2$ page. That is, we show the following: 

\begin{prop}\label{e2iso} The map induced on the $E_2$ page by the morphism of spectral sequences in \autoref{morphism} is an isomorphism. Hence by Lemma 2.5 the induced map on the $E_r$ page for every $r\geq 2$ also is an isomorphism.
\end{prop}

The general approach is to localize, demonstrate an isomorphism on the localizations through induction on the difference in dimension, and then glue the local isomorphisms together.  Again, let $Y \rightarrow X$ be an embedding of $Y$ into an arbitrary smooth affine scheme $X$, and let $X \rightarrow \mathbb A^n$ be an embedding of $X$ as a closed subscheme of $\mathbb{A}^n$.  Let $S$, $R$ and $T$ be the coordinate rings associated with these varieties.  That is, $Y= \Spec S,\ X=\Spec R$, and $T=\kk[x_1,...,x_n]$.  Let $I$ be the kernel of the surjection $T \rightarrow S$, and let $J$ be the kernel of the surjection $T \rightarrow R$.  

Since $X$ is a smooth scheme, $\Omega_{R/\kk}$ is locally free of rank $s=\dim X$.  Thus we may find $f_1, \dots ,f_m \in R$ so that the $\Spec R_{f_i}$ form an open cover of $\Spec R$, and so that for each $f_i$ there exist $y_{i1},\dots,y_{is} \in R$ so $dy_{i1},\dots,dy_{is}$ generate $\Omega_{R_{f_i}/\kk}$ as a $T_{f_i}$-module.

Write $Y_{f_i}$ for $\Spec S_{f_i}$ and $X_{f_i}$ for $\Spec R_{f_i}$.  We show our claim first for the local case $Y_{f_i} \rightarrow X_{f_i}$ 
and then show how to extend this to the situation $Y \rightarrow X$.

Fix some $f_i$ and say $\Omega_{R_{f_i}/\kk}$ is generated by $dy_1,...,dy_s$. We know there is an exact sequence for differentials over $\kk$ from \autoref{conormal}:

$$J_{f_i}/J_{f_i}^2 \rightarrow R_{f_i} \otimes_{T_{f_i}} \Omega_{T_{f_i}/\kk} \rightarrow \Omega_{R_{f_i}/\kk}\rightarrow 0,$$
which shows that we can extend the set of generators of $\Omega_{R_{f_i}/\kk}$ to a generating set of $\Omega_{T_{f_i}/\kk}$ by picking $n-s$ elements $dz_1,...,dz_{n-s}$ with $J_{f_i}=(z_1\ ...\ z_{n-s})$.

 \begin{lemma} Notation being as above, we may view the $J_{f_i}$-adic completion of $T_{f_i}$, which we will denote $\hat{T}_{f_i}$, as a power series ring $\hat{T}_{f_i} \cong (T_{f_i}/J)[[z_1,\dots, z_{n-s}]]=R_{f_i}[[z_1,\dots,z_{n-s}]]$.
 \end{lemma}
 
 \begin{proof}
 Since $R_{f_i}$ is formally smooth over $\kk$, there exists a ring homomorphism $\phi:R_{f_i}\to \hat T_{f_i}$ such that the composition of $\phi$ with the natural surjection $\hat T_{f_i}\to \hat T_{f_i}/J$ is the identity map $R_{f_i}\to \hat T_{f_i}/J=R_{f_i}$. Denote $\phi(R_{f_i})$ by $R_{f_i}$. Since $\hat T_{f_i}$ is complete and $J$ is an ideal of height $n-s$ generated by $z_1,\dots,z_{n-s}$, it follows that   $T_{f_i}=R_{f_i}[[z_1,\dots,z_{n-s}]]$.
 \end{proof}

\begin{lemma}\label{completion} The modules in the complex $H^{\bullet}_I(T_{f_i})$ have a natural structure as $\hat{T}_{f_i}$-modules.
 \end{lemma}
\begin{proof}   Since $I \subset J$, every element of those modules is killed by a power of $J$, so each inherits a $\hat{T}_{f_i}$-module structure from its $T_{f_i}$ structure.  It's well-known that local cohomology commutes with completions, and it follows that each module is isomorphic to the local cohomology of the completion as a $\hat{T}_{f_i}$-module.
\end{proof}

At this point we need a definition of the plus operation from \cite{emilyluis} (this definition is motivated by the celebrated Kashiwara equivalence theorem):  

\begin{definition}
Let $A$ be a noetherian ring, and let $B=A[[z]]$.  The plus operation, is the functor from the category of $A$-modules to the category of $B$-modules given by $M_+ = M\otimes_A(B_z/B)=M\cdot\frac{1}{z}\oplus M\cdot\frac{1}{z^2}\oplus\dots\oplus M\cdot \frac{1}{z^i}\oplus\dots.$ 
\end{definition}

Multiplication by $z$ snds $\frac{m}{z^i}$ to $\frac{m}{z^{i-1}}$. Thus the annihilator of $z$ in $M_+$ is the $A$-module $M\cdot\frac{1}{z}$. If $A$ is a regular finitely generated $\kk$-algebra and $M$ is a $D(A,\kk)$-module, then $M_+$ is a $D(B,\kk)$-module (namely, $\partial_z$ sends $\frac{m}{z^i}$ to $(-i)\frac{m}{z^{i+1}}$).

\autoref{tmap} says that $\Omega_{B/\kk}=(B\otimes_A\Omega_{A/\kk})\oplus P$ and in this particular case $P=Bdz$, the free module of rank one generated by $dz$.
Hence there is a chain map of de Rham complexes $\phi^\bullet:\Omega^\bullet (M)\to \Omega^\bullet (M_+)$, where $\Omega^\bullet(M)$ is the de Rham complex of $M$ in the category of $A$-modules while $\Omega^\bullet(M_+)$ is the de Rham complex of $M_+$ in the category of $B$-modules and the chain map sends $m\otimes\omega\in M\otimes_A\Omega^i(A)$ to $(m\otimes\omega)\cdot\frac{1}{z}dz\in M_+\otimes_B\Omega^{i+1}(B)$. Thus the chain map sends $\Omega^\bullet(M)$ to the annihilator of the ideal $J=(z)$ in the complex $P\otimes_B\wedge^\bullet(B\otimes_A\Omega_{A/\kk})$, exactly as in the proof of \autoref{morphism}.

\begin{prop}\label{isohom}
This chain map $\phi^\bullet:\Omega^\bullet(M)\to \Omega^\bullet(M_+)$ induces an isomorphism on homology, i.e.~$\phi^i_*:h^p(\Omega^\bullet(M))\to h^{p+1}(\Omega^\bullet(M_+))$ is an isomorphism for all $p$.
\end{prop}

\emph{Proof.} Consider the short exact sequence of modules
\begin{equation}\label{ses} 0 \rightarrow M _+\otimes_A \Omega^{p-1}(A) \rightarrow M_+ \otimes_B \Omega^p (B) \rightarrow M_+\otimes_A \Omega^p (A) \rightarrow 0,
\end{equation}
 where the first map is the wedge map with $dz$, and the second map acts by sending $dz$ to $0$.  This induces a short exact sequence on complexes, and so induces a long exact sequence in de Rham cohomology:
$$\cdots \rightarrow h^{p-1}(M_+ \otimes_A \Omega^{\bullet}(A)) \rightarrow h^p(M_+ \otimes_B\Omega^{\bullet}(B)) \rightarrow h^p (M_+ \otimes_{A} \Omega^{\bullet}(A)) \rightarrow h^p (M_+ \otimes_A \Omega^{\bullet}(A)) \rightarrow \cdots$$ where the last map is the connecting homomorphism.  Switala showed that the connecting homomorphism map is the same as $\partial_z$. (Lemma 2.22 in \cite{switala})

Hence the last map above is $\partial z: h^p (M_+ \otimes_{A} \Omega^{\bullet}(A)) \rightarrow h^p (M_+ \otimes_A \Omega^{\bullet}(A))$.  Since the maps in the complex $\Omega^\bullet(A)$ are unaffected by $z$ or $dz$, it follows that $$h^p (M_+ \otimes_{A} \Omega^{\bullet}(A))=h^p(M\otimes_A\Omega^\bullet(A))\cdot\frac{1}{z}\oplus h^p(M\otimes_A\Omega^\bullet(A))\cdot\frac{1}{z^2}\oplus\dots.$$ Therefore $$\ker(\partial z: h^p (M_+ \otimes_{A} \Omega^{\bullet}(A)) \rightarrow h^p (M_+ \otimes_A \Omega^{\bullet}(A)))= 0$$ and $$\coker(\partial z: h^p (M_+ \otimes_{A} \Omega^{\bullet}(A)) \rightarrow h^p (M_+ \otimes_A \Omega^{\bullet}(A)))= h^p(M\otimes_A\Omega^\bullet(A))\cdot\frac{1}{z}.$$ 

Therefore $h^{p+1}(\Omega^\bullet(M_+))=h^p(M\otimes_A\Omega^\bullet(A))$ and the isomorphism is induced by the chain map $M\otimes_A\Omega^\bullet(A)\to M_+\otimes_B\Omega^\bullet(B)$ that sends $M\otimes_A\Omega^i(A)$ to $M\otimes_A\Omega^i(A)\cdot\frac{1}{z}dz\in \Omega^{p+1}(B)$, exactly as claimed. \qed

Nu\'{n}\~{e}z-Betancourt and Witt showed the following:

\begin{lemma}[Lemma 3.9 in \cite{emilyluis}]\label{plusop}
Let $A$ be a noetherian ring, $M$ an $A$-module, and $B=A[[x]]$.  For every ideal $I \subseteq A$ any $i \in \mathbb N$, $$(H^i_I(M))_+ \cong H^{i+1}_{(I,x)}(M\otimes_A B)$$
\end{lemma}

\begin{prop}
Let $f_i\in T$ be such that the ideal $J_{f_i}$ is a complete intersection, i.e.~$J_{f_i}=(z_1,\dots,z_{n-s})$. Let $\psi^\bullet:\Omega^\bullet(R_{f_i})\to \Omega^\bullet(H^{n-s}_J(T_{f_i}))$ be the morphism of complexes from the proof of \autoref{morphism} and let $H^q_I(\psi^\bullet):\Omega^\bullet(H^q_I(R_{f_i}))\to \Omega^\bullet(H^q_I(H^{n-s}_J(T_{f_i})))$ be the induced morphism on the de Rham complexes of local cohomology modules. Considering that $H^q_I(H^{n-s}_J(T_{f_i}))\cong H^{q+n-s}_I(T_{f_i})$, the morphism $H^{n-s}_I(\psi^\bullet)$ induces an isomorphism on homology, i.e.~ the map $\psi^p_*:h^p(\Omega^\bullet(H^q_I(R_{f_i})))\to h^{p+n-s}(\Omega^\bullet(H^{q+n-s}_{I}(T_{f_i})))$ is an isomorphism.
\end{prop}

\emph{Proof.} According to \autoref{completion}, after renaming $R_{f_i}$ as $A$ we may replace $T_{f_i}$ by the power series ring $B=A[[z_1,\dots,z_{n-s}]]$.
Now let $B_t=A[[z_1,\dots,z_t]]$ for $1\leq t\leq k$. Then $B_{t}=B_{t-1}[[z_{t}]]$ and therefore iterating the plus construction $t$ times we get by induction on $t$ that $(\dots(H^i(A)_+)_+\dots)_+=H^{i+t}_{(I, z_1,\dots,z_t)}(B_t)$.

The morphism of complexes $H^{n-s}_I(\psi^\bullet):\Omega^\bullet(H^{n-s}_I(A))\to \Omega^\bullet(H^{n-s}_{(I,z_1,\dots,z_{n-s})}(B))$ is the composition $\psi_1^\bullet\circ\psi_2^\bullet\circ\dots\circ\psi^\bullet_{n-s}$, where $\psi_i^\bullet:\Omega^\bullet(H^{i-1}_{(I,z_1,\dots,z_{i-1})}(B_{i-1}))\to \Omega^\bullet (H^i_{(I,z_1,\dots,z_i)}(B_i))$ is the morphism of \autoref{completion}. Since by \autoref{completion} each morphism in the composition induces an isomorphism on homology, so does the composition. \qed

We can now prove \autoref{e2iso}:

\emph{Proof.} Let $\psi^\bullet:\Omega^\bullet(R)\to \Omega^\bullet(H^{n-s}_J(T))$ be the morphism of complexes in \autoref{morphism}. It induces a morphism $\psi^\bullet_{f_i}:\Omega^\bullet(R_{f_i})\to \Omega^\bullet(H^{n-s}_J(T_{f_i}))$ for every $i$. By \autoref{isohom} this map induces an isomorphism $(\psi_{f_i}^p)_*:h^p(\Omega^\bullet(H^q_I(R_{f_i})))\to h^p(\Omega^\bullet(H^{q+n-s}_I(T_{f_i})))$ on homology, for all $p$. 

For any $R$-module or $T$-module $M$, the Cech complex: $$0 \rightarrow M \rightarrow \oplus M_{f_i} \rightarrow \dots \rightarrow M_{f_1\dots f_m} \rightarrow 0$$ is exact. Hence the induced morphism of complexes $H^q_I(\psi^\bullet):\Omega^\bullet(H^q_I(R))\to \Omega^\bullet(H^q_I(H^{n-s}_J(T)))$ induces (considering that $H^q_I(H^{n-s}_J(T))\cong H^{q+n-s}_I(T)$, a commutative diagram of complexes with exact rows:
\begin{scriptsize}
$$\begin{CD}
0@>>>\Omega^\bullet(H^q_I(R))@>>>\oplus\Omega^\bullet(H^q_I(R_{f_i}))@>>>\dots@>>>\Omega^\bullet(H^q_I(R_{f_1\cdots f_m}))@>>>0\\
@.@VV\psi^\bullet V@VV\oplus\psi^\bullet_{f_i}V@.@VV\psi^\bullet_{f_1\cdots f_m}V\\
0@>>>\Omega^\bullet(H^{q+n-s}_I(T))@>>>\oplus\Omega^\bullet(H^{q+n-s}_I(T)_{f_i})@>>>\dots@>>>\Omega^\bullet(H^{q+n-s}_I(T)_{f_1\cdots f_m})@>>>0\\
\end{CD}$$
\end{scriptsize}
where the vertical maps $\oplus\psi_{f_{i_1}\cdots f_{i_j}}^\bullet$ are direct sums of localizations of the vertical map $\psi^\bullet$ on the very left of the diagram. We have shown that each $\psi^\bullet_{{f_1}\cdots f_{i_j}}$ induces an isomorphism on homology. Hence so does every vertical map $\oplus\psi_{f_{i_1}\cdots f_{i_j}}^\bullet$.

Let $K^\bullet_{R,j}$ be the kernel of the chain map $\oplus\Omega^\bullet(H^q_I(R)_{f_{i_1}\cdots f_{i_j}}))\to \oplus\Omega^\bullet(H^q_I(R_{f_{i_1}\cdots f_{i_j}f_{i_{j+1}}}))$ in the top row of the diagram and let $K^\bullet_{T,j}$ be the kernel of the corresponding chain map $\oplus\Omega^\bullet(H^{q+n-s}_I(T_{f_{i_1}\cdots f_{i_j}}))\to \oplus\Omega^\bullet(H^{q+n-s}_I(T_{f_{i_1}\cdots f_{i_j}f_{i_{j+1}}}))$ in the bottom row. Both $K^\bullet_{R,j}$ and $K^\bullet_{T,j}$ are complexes and the vertical maps in the diagram induce a morphism of complexes $\psi^\bullet_{K,j}:K^\bullet(R,j)\to K^\bullet(T,j)$. 

We claim that the induced maps on homology $(\psi^p_{K,j})_*:h^p(K^\bullet(R,j))\to h^p(K^\bullet(T,j))$ are isomorphisms for all $p$. To prove this claim we use descending induction on $j$. To begin the induction, let $j=m-1$. The above diagram induces the following commutative diagram with short exact rows.

$$\minCDarrowwidth10pt\begin{CD}
0@>>>K^\bullet_{R,m-1}@>>>\oplus_j\Omega^\bullet(H^q_I(R_{{f_1}\cdots f_{i_{j-1}}f_{i_{j+1}}\cdots f_m}))@>>> \Omega^\bullet(H^q_I(R_{{f_1\cdots f_m}}))@>>> 0\\
@. @VV\psi^\bullet_{K,m-1}V @VV\oplus_j\psi^\bullet_{1,\dots,j-1,j+1,\dots,m}V@VV\psi^\bullet_{{f_1}\cdots f_m}V\\
0@>>>K^\bullet_{T,m-1}@>>>\oplus_j\Omega^\bullet(H^{q+n-s}_I(T_{{f_1}\cdots f_{i_{j-1}}f_{i_{j+1}}\cdots f_m}))@>>> \Omega^\bullet(H^{q+n-s}_I(T_{{f_1\cdots f_m}}))@>>> 0\\
\end{CD}$$

Let us write, for the sake of legibility, $A_R^{\bullet} = \oplus_j\Omega^\bullet(H^q_I(R_{{f_1}\cdots f_{i_{j-1}}f_{i_{j+1}}\cdots f_m}))$, $A_T^{\bullet}$ = $\oplus_j\Omega^\bullet(H^{q+n-s}_I(T_{{f_1}\cdots f_{i_{j-1}}f_{i_{j+1}}\cdots f_m}))$, and $p' = p+n-s$. Then this commutative diagram induces the following commutative diagram where the rows are the long exact sequences produced by the short exact sequences in the rows of the preceding diagram.  

\begin{footnotesize}
$$\minCDarrowwidth10pt\begin{CD}
h^{p-1}(A_R^{\bullet}) @>>> h^{p-1}(\Omega^\bullet(R_{{f_1\cdots f_m}})) @>>> h^p(K^\bullet_{R,m-1})@>>> h^{p}(A_R^{\bullet})@>>> h^{p}(\Omega^\bullet(R_{{f_1\cdots f_m}}))\\
@VV\cong V @VV\cong V @VVV @VV \cong V @VV\cong V\\
h^{p'-1}(A_T^{\bullet}) @>>> h^{p'-1}(\Omega^\bullet(H^{n-s}_J(T_{{f_1\cdots f_m}})))@>>> h^{p'}(K^\bullet_{T,m-1}) @>>> h^{p'}(A_T^{\bullet})@>>> h^{p'}(\Omega^\bullet(H^{n-s}_J(T_{{f_1\cdots f_m}}))) \\
\end{CD}$$
\end{footnotesize}

The two vertical maps on the left and the two vertical maps on the right have been shown to be isomorphisms. By the 5-lemma the map in the middle is an isomorphism as well. This finishes the proof of the claim in the case $j=m-1$. 

Now let $j<m-1$ and assume the claim proven for $j+1$. Since the rows in the very first commutative diagram in this proof are exact, we get the following commutative diagram with exact rows, where the direct sum in the middle is over all $j$-element subsets of the set $\{1,\dots,m\}$.

$$\minCDarrowwidth10pt\begin{CD}
0@>>>K^\bullet_{R,j}@>>>\oplus\Omega^\bullet(R_{{f_{i_1}}\cdots f_{i_j}})@>>> K^\bullet_{R,j+1}@>>> 0\\
@. @VV\psi^\bullet_{K,j}V @VV\oplus\psi^\bullet_{f_{i_1}\cdots f_{i_j}}V@VV\psi^\bullet_{K,j+1}V\\
0@>>>K^\bullet_{T,j}@>>>\oplus\Omega^\bullet(H^{n-s}_J(T_{f_{i_1}\cdots f_{i_j}}))@>>> K^\bullet_{T,j+1}@>>> 0\\
\end{CD}$$

This commutative diagram induces the following commutative diagram where the rows are the long exact sequences produced by the short exact sequences in the rows of the preceding diagram.  Again, we write $p'$ for $p+n-s$.

\begin{footnotesize}
$$\minCDarrowwidth10pt\begin{CD}
h^{p-1}(\oplus\Omega^\bullet(R_{{f_{i_1}}\cdots f_{i_j}}) @>>> h^{p-1}(K^\bullet_{R,j+1}) @>>> h^p(K^\bullet_{R,j})@>>> h^{p}(\Omega^\bullet(R_{{f_{i_1}}\cdots f_{i_j}})@>>> h^{p}(K^\bullet_{R,j+1})\\
@VV\cong V @VV\cong V @VVV @VV \cong V @VV\cong V\\
h^{p'-1}( \oplus\Omega^\bullet(H^{n-s}_J(T_{f_{i_1}\cdots f_{i_j}})) @>>> h^{p'-1}(K^\bullet_{T,j+1})@>>> h^{p'}(K^\bullet_{T,j}) @>>> h^{p'}(\oplus\Omega^\bullet(H^{n-s}_J(T_{f_{i_1}\cdots f_{i_j}}))@>>> h^{p'}(K^\bullet_{T,j+1}) \\
\end{CD}$$
\end{footnotesize}

The first and the third vertical maps in this diagram have been shown to be isomorphisms, while the second and the foruth vertical maps are isomorphisms by the induction hypothesis. The 5-lemma implies that the vertical map in the middle is an isomorphism as well. This completes the proof of the claim.

In particular, the map $\psi^\bullet_{K,1}:K^\bullet(R,1)\to K^\bullet(T,1)$ induces for all $p$ isomorphisms $(\psi^p_{K,1})_*:h^p(K^\bullet(R,1))\to h^p(K^\bullet(T,1))$ on homology.

The very first commutative diagram in this proof induces the following commutative diagram with short exact rows.

$$\minCDarrowwidth10pt\begin{CD}
0@>>>\Omega^\bullet(R)@>>>\oplus\Omega^\bullet(R_{f_i})@>>>K^\bullet_{R,1}@>>>0\\
@.@VV\psi^\bullet V@VV\oplus\psi^\bullet_{f_i}V@VV\psi^\bullet_{K,1}V\\
0@>>>\Omega^\bullet(H^{n-s}_J(T))@>>>\oplus\Omega^\bullet(H^{n-s}_J(T)_{f_i})@>>>K^\bullet_{T,1}@>>>0\\
\end{CD}$$

This diagram similarly produces a commutative diagram with long exact sequences in the rows and the 5-lemma then implies that the chain map $\psi^\bullet$ induces isomorphisms on homology.\qed

\section{Proof of Theorem 1.1.}

Let $Y\to X$ be an embedding of an affine variety $Y$ over a characteristic 0 field $\kk$ into a smooth affine variety $X$ over $\kk$. Theorem 1.1 says that 

(1.1a) Starting with the $E_2$ page, the spectral sequence $$E_1^{n-p,n-q}=H^{n-q}_Y(X, \Omega_X^{n-p})\Longrightarrow H^{\rm dR}_{p+q}(Y)$$ is independent of the choice of $X$ and the embedding, up to a bi-degree shift, and 

(1.1b) The modules appearing on the $E_2$ page, and hence on every $E_r$ page for $r\geq 2$, are finite-dimensional $\kk$-spaces.

\smallskip

\emph{Proof of 1.1a.} Let $Y\to X'$ and $Y\to X''$ be two embeddings of $Y$ into non-singular affine varieties $X'$ and $X''$. Emded $X'$ and $X''$ into affine spaces $\mathbb A^{n'}$ and $\mathbb A^{n''}$ respectively. The compositions $Y\to X'\to \mathbb A^{n'}$ and $Y\to X''\to \mathbb A^{n''}$ induce embeddings of $Y$ into $\mathbb A^{n'}$ and $\mathbb A^{n''}$. 

Set $T'=\kk[x_1,\dots,x_{n'}]$ and $T''=\kk[y_1,\dots,y_{n''}]$ so that $\mathbb A^{n'}={\rm Spec}T'$ and $\mathbb A^{n''}={\rm Spec}T''$. Let $S$ be the coordinate ring of $Y$ and let our embeddings $Y\to \mathbb A^{n'}$ and $Y\to \mathbb A^{n''}$ be given by surjections $\phi':T'\to S$ and $\phi'':T''\to S$ respectively. The diagonal embedding $Y\to \mathbb A^{n'}\otimes_{{\rm Spec}k}\mathbb A^{n''}=\mathbb A^{n'+n''}$ is then given by the surjection $\phi=\phi'\otimes_{\kk}\phi'':T\stackrel{\rm def}{=}T'\otimes_{\kk}T''\to S$ that sends $r'\otimes_{\kk}  r''$ to $\phi'(r')\cdot \phi''(r'')$.

For every $i$ let $f_i\in T'$ be an element such that $\phi''(y_i)=\phi'(f_i)$ ($f_i$ exists because the map $\phi'$ is surjective). Since $T$ is the ring of polynomials in the $n'+n''$ variables $x_i\otimes 1, 1\otimes y_j$ where $1\leq i\leq n', 1\leq j\leq n''$, there is a $\kk$-algebra surjection $\psi':T\to T'$ that sends every $x_i\otimes 1$ to $x_i$ and every $1\otimes y_j$ to $f_j$ (note that $\psi'$ is not canonical because $f_i$ is not uniquely defined). Hence the composition $T\stackrel{\psi'}{\to}T'\stackrel{\phi'}{\to}S$ is the diagonal surjection $\phi:T\to S$.

We get a sequence of three closed embeddings $Y\to X\to \mathbb A^{n'}\to \mathbb A^{n'+n''}$, where the first two have been described in the first paragraph of this proof, the last one is induced by the surjection $\psi':T\to T'$, and the composition of all three embeddings in the sequence is the diagonal embedding $Y\to \mathbb A^{n'+n''}$. Applying Proposition 3.1 and Proposition 4.1 to the resulting nested embedding $Y\to X'\to \mathbb A^{n'+n''}$ we conclude that the spectral sequence of the embedding $Y\to X'$ is isomorphic, starting with the  $E_2$ page and up to a suitable bi-degree shift, to the spectral sequence of the diagonal embedding $Y\to\mathbb A^{n'+m''}$. A symmetric argument shows that the spectral sequence of the embedding $Y\to X''$ also is isomorphic, starting with the  $E_2$ page and up to a suitable bi-degree shift, to the spectral sequence of the diagonal embedding $Y\to\mathbb A^{n'+m''}$. Thus the two embeddings $Y\to X'$ and $Y\to X''$ produce spectral sequences that are isomorphic to the same spectral sequence (starting with the $E_2$ page and up to suitable bi-degree shifts). Hence these two spectral sequences are isomorphic to each other. This concludes our proof of 1.1a. \qed

\smallskip

\emph{Proof of 1.1b.}
Due to the embedding-independence shown before, we may consider the case $X=\Spec k[x_1,...,x_n]$, and write $T$ for $\kk[x_1,...,x_n]$.  Write $S$ for the ring with $Y = \Spec S$, and let $I$ be the kernel of the surjection $T \rightarrow S$.
The $E_1$ page of the embedding is given by $E_1^{p,q} = H_{Y}^q(X,\Omega^p_{X})$, or equivalently, $E_1^{p,q}=H^q_{I}(T\otimes\Omega^p(T))$.  Since $T$ is the ring of polynomials in $n$ variables, $\Omega^p(T)$ is free of rank $n$.  Local cohomology commutes with direct sums, so this is the same as 
$H_{I}(T)\otimes_{T}\Omega^p(T)$.  Now as before, the maps in any row are the de Rham differentials of the complex $H^q_{I}(T)\otimes_{T}\Omega^{\bullet}$(T), so the modules on the $E_2$ page - the cohomology objects of these rows - are finite-dimensional by \autoref{findim}, because $H^q_{I}(T)$ is a holonomic $\mathcal{D}(T,\kk)$-module. This completes our proof of 1.1b and Theorem 1.1. \qed

\end{document}